\documentclass[a4paper,11pt,reqno]{amsart}

\usepackage{
a4wide,
amsfonts,
amssymb,
amscd,
amsmath,
latexsym,
amsbsy,
enumerate
}

\newtheorem{thm}{Theorem}[section]
\newtheorem{defn}[thm]{Definition}
\newtheorem{lem}[thm]{Lemma}
\newtheorem{cor}[thm]{Corollary}
\newtheorem{prop}[thm]{Proposition}

\theoremstyle{definition}
\newtheorem{rmk}[thm]{Remark}

\numberwithin{equation}{section}
\def\al{\alpha}
\def\be{\beta}
\def\ga{\gamma}
\def\de{\delta}
\def\la{\lambda}
\def\ep{\varepsilon}
\def\io{\iota}
\def\si{\sigma}

\def\Om{\Omega}
\def\R{\mathbb{R}}
\def\C{\mathbb{C}}
\def\N{\mathbb{N}}
\def\Z{\mathbb{Z}}

\def\bm{\mathbf{m}}
\def\bk{\mathbf{k}}
\def\bn{\mathbf{n}}
\def\bl{\mathbf{l}}
\def\Id{\mathrm{Id}}
\def\SL{\mathrm{SL}}
\def\SU{\mathrm{SU}}
\def\diag{\textrm{diag}}
\def\root{\textrm{root}}
\def\Eroot{\textrm{E-root}}
\def\Froot{\textrm{F-root}}
\def\Ker{\textrm{Ker}}
\def\Tr{\textrm{Tr}}

\newcommand{\Lg}{\mathfrak{g}}
\newcommand{\Lh}{\mathfrak{h}}
\newcommand{\Lb}{\mathfrak{b}}
\newcommand{\Lsl}{\mathfrak{sl}}
\newcommand{\Lsu}{\mathfrak{su}}
\newcommand{\qbinomial}[3]{\left[\genfrac{.}{.}{0pt}{}{#1}{#2}\right]_{#3}}

\title[Non-extremal weight modules]{Non-extremal weight modules for quantized universal 
enveloping algebras}

 \author{Erik Koelink}
 \address{IMAPP, Radboud Universiteit, PO Box 9010,
6500 GL Nijmegen,
 the Netherlands}
 \email{e.koelink@math.ru.nl, henrique.t.tavares@hotmail.com}
 \author{Henrique Tyrrell}

 \date{\today}
\begin{document}

\begin{abstract} 
For quantized universal enveloping algebras we construct weight modules by inducing representations
of the centralizer of the Cartan subalgebra in the quantized universal enveloping algebra. 
The induced modules arising from finite-dimensional weight modules the centralizer algebra are studied.
In particular, 
we study the induction of one-dimensional modules, and this is related to the study of 
commutative subalgebras of the centralizer algebra.
For the special case of $U_q(\Lsl(2,\C))$ we
show that we get the admissible unitary representations corresponding to the non-compact real 
form $U_q(\Lsu(1,1))$. 
\end{abstract}

\maketitle


\section{Introduction}\label{sec:intro}

Large classes of 
representations of quantized universal enveloping algebras $U_q(\Lg)$ for simple complex Lie algebras, such
as finite dimensional representations or Verma modules, see e.g. \cite{Dixm} for 
the classical case, are well understood, see e.g. 
\cite{CharP}, \cite{Jant}, \cite{Jose}, \cite{KlimS}. 
On the other hand, these representations (or modules) do not suffice for the harmonic analysis on quantum analogs of  
non-compact quantum groups. The best known example of an analytically studied non-compact quantum group
is the quantum analog of the universal enveloping algebra of $\Lsu(1,1)$. 
The irreducible $\ast$-representations have been classified by
Vaksman and Korogodski\u\i\ \cite{VaksK}, by Burban and Klimyk \cite{BurbK} and by Masuda et al. \cite{MasuMNNSU}.
In this case we see that the representation theory of $U_q(\Lsu(1,1))$ differs from the irreducible 
unitary representations of the Lie algebra $\Lsu(1,1)$. 
The so-called strange series representations do not have a classical analog; they formally vanish in the 
limit $q\to 1$. 
It turns out that in the analytic 
study of this non-compact quantum group these representations play an important role, see 
\cite{GroeKK}, \cite{KoelS}, \cite{Vaks}, \cite{VaksK} and references given there. 
The representations that play a role in this example are non-extremal weight representations, i.e. these representations are weight representations that have neither a highest weight nor a lowest weight. 
In this paper we present another way to obtain these representations.

The idea is to use the centralizer of the analog $U^0$ of the Cartan subalgebra of $U_q(\Lg)=U^-\otimes U^0\otimes U^+$, i.e. 
the trivial weight space in the weight decomposition $U_q(\Lg)=\bigoplus_{\be\in Q} U_q(\Lg)_\be$, 
where $Q$ is the corresponding root lattice, see Section \ref{ssec:notconv} for notation. 
We then construct weight representations of $U_q(\Lg)$ by inducing a weight module of the centralizer algebra 
$U_0$. The construction is called Mathieu module, being inspired by the paper \cite{Math} by Mathieu on 
the study of weight modules for Lie algebras. In Mathieu's paper \cite{Math} the parabolic induction 
is the key procedure, and in Futorny et al. \cite{FutoKZ} a quantum analogue for $\Lsl(n,\C)$ is given. 

In particular, we are interested in the case of the induction 
of $1$-dimensional modules of the centralizer algebra $U_0$. In order to do so, we look for commutative 
subalgebras of the centralizer algebra $U_0$, which is closely related to strongly orthogonal roots, 
see \cite{AgaoK}, \cite{Letz}. We show that for the case of $\Lg=\Lsl(2,\C)$ and for the 
$\ast$-structure for the non-compact real form $U_q(\Lsu(1,1))$ we recover the representations 
of \cite{BurbK}, \cite{MasuMNNSU}, \cite{VaksK}. 

In Section \ref{sec:centralizerCSA} we introduce and study the centralizer algebra $U_0$ using 
the PBW-basis and suitable height functions. We discuss commutative 
subalgebras of $U_0$ in relation to strongly orthogonal roots. 
In Section \ref{sec:Mathiemod} we introduce the 
induced representations, which we call Mathieu modules. In Section \ref{sec:Mathieumodrank1} 
we focus our attention on the induction of $1$-dimensional representations. We study the simplest case $\Lg=\Lsl(2,\C)$ in Section \ref{sec:Mathieumodsl2C}.
In Section \ref{sec:Mathieumodsln+1C} we discuss some aspects of this construction for 
$\Lg=\Lsl(n+1,\C)$. 

We expect that the non-extremal weight modules constructed in this way can be used to 
improve the understanding of the 
harmonic analysis of non-compact quantum groups, see \cite{Vaks}.

\medskip
\emph{Acknowledgement.} We thank Maarten van Pruijssen for useful discussions 
and suggestions. We also thank Kenny De Commer for discussions. 
H.~Tyrrell thanks Pablo Rom\'an and Universidad Nacional de C\'ordoba for 
its hospitality. The research of H.~Tyrrell is supported by the Brazilian agency CNPq, 
Conselho Nacional de Pesquisa e Desenvolvimento, 200678-2015/9.


\subsection{Notation and conventions}\label{ssec:notconv}

We use the notation $\N=\{1,2,3,\cdots\}$ and we use $\N_0$ for the set $\{0,1,2,3,\cdots\}$.

We use the conventions and notations for quantized universal enveloping algebras 
as in \cite{KlimS}, see also e.g.  
\cite{CharP}, \cite{Jant}. All statements in the this section can be found in \cite{KlimS}. 

Let $\Lg$ be a complex semi-simple Lie algebra with a Cartan subalgebra $\Lh$ and 
$\Phi$ be the corresponding root system. Let $n=\textrm{rank}\, \Lg$ and fix the 
simple roots $\Pi=\{\al_1,\cdots,\al_n\}$. Let $\Phi^+$ be the set of positive roots
and set  $r = |\Phi^+|$.  By $Q=\bigoplus_{i=1}^n \Z \al_i\subset \Lh^\ast$ we denote the root lattice and 
$Q^+=\bigoplus_{i=1}^n \N_0 \al_i$ denotes the corresponding positive roots. 
The Cartan matrix is  $A=(a_{i,j})_{i,j=1}^n$. Let $D=\diag (d_1,\cdots, d_n)$ be the diagonal matrix 
so that $d_i\in\{1,2,3\}$ and $DA$ is symmetric and positive definite. 
Let $(\cdot, \cdot)$ be the corresponding bilinear form on $\Lh^\ast$. 

We consider $q$ as a non-zero element of $\C$, and we assume $q$ is not a root of unity. 
We let $q_i=q^{d_i}$ and we use the $q$-binomial coefficient for  $n,k\in\N_0$ with $0\leq k\leq n$;
\begin{equation*}
\qbinomial{n}{k}{q} = \frac{[n]_q!}{[k]_q^!\, [n-k]_q!},
\qquad [k]_q! = \prod_{j=1}^k [j]_q, \qquad [j]_q= \frac{q^j-q^{-j}}{q-q^{-1}}
\end{equation*}

\begin{defn}\label{def:QUE}
The quantized enveloping algebra $U = U_q(\Lg)$ is the 
unital associative algebra generated by elements $E_i, F_i, K_i, K_i^{-1}$, 
$i \in \{1,\cdots, n\}$, subject to the relations:
\begin{gather*}
K_iK_i^{-1}=1=K_i^{-1}K_i, \qquad K_iK_j=K_jK_i, \\
K_iE_jK_i^{-1} = q_i^{a_{i,j}}E_j, \qquad 
K_iF_jK_i^{-1} = q_i^{-a_{i,j}}F_j, \qquad
E_iF_j - F_jE_i = \de_{i,j}\frac{K_i - K_i^{-1}}{q_i - q_i^{-1}}, \\
\sum_{r = 0}^{1 - a_{ij}} (-1)^r \qbinomial{1-a_{i,j}}{k}{q_i}
E_i^{1 - a_{ij} - r}E_jE_i^r = 0, \quad i\not= j, \\
\sum_{r = 0}^{1 - a_{ij}} (-1)^r \qbinomial{1-a_{i,j}}{k}{q_i}
F_i^{1 - a_{ij} - r}F_jF_i^r = 0, \quad i\not= j.
\end{gather*}
\end{defn}

Note that $q_i^{a_{i,j}}=q^{(\al_i,\al_j)}$. 
The last two relations in Definition \ref{def:QUE} are known as the 
$q$-analogs of the Serre relations.

Denote by $U^+=U_q(\mathfrak{n}^+)$ the subalgebra generated by $E_i$, $1\leq i \leq n$, and
similarly we let $U^-=U_q(\mathfrak{n}^-)$ be the subalgebra generated by $F_i$, $1\leq i \leq n$, which 
are the analogues of the universal enveloping algebra for the subalgebras $\mathfrak{n}^\pm$ 
in the decomposition $\Lg=\mathfrak{n}^-\oplus\Lh\oplus \mathfrak{n}^+$. 
Put 
$U^0$ for the subalgebra generated by $K_i^\pm$, $1\leq i \leq n$.  Then the multiplication map
\begin{equation*}
U^+ \otimes U^0 \otimes U^- \to U
\end{equation*}
is an isomorphism of vector spaces. 

In order to describe the PBW (Poincar\'e-Birkhoff-Witt) basis of $U=U_q(\Lg)$ we fix a 
reduced decomposition $w_0 =s_{i_1}\cdots s_{i_r}$ of the longest Weyl group element $w_0\in W$ 
in terms of the reflections $s_i$ corresponding to the simple root $\al_i$.
Then $\be_1=\al_{i_1}$, $\be_2=s_{i_1}(\al_{i_2})$, ..., 
$\be_r=s_{i_1}\cdots s_{i_{r-1}}(\al_{i_r})$ exhaust the positive roots $\Phi^+$. 
In the quantum case, there exist elements $T_i$, $1\leq i\leq n$, satisfying the 
braid relations for $\Lg$, and the root vectors $E_{\be_r}$, $F_{\be_r}$ are defined as
\begin{equation*}
E_{\be_r} = T_{i_1}\cdots T_{i_{r-1}}(E_{i_r}), \qquad
F_{\be_r} = T_{i_1}\cdots T_{i_{r-1}}(F_{i_r}).  
\end{equation*}
Now the PBW basis for $U$ is given by
\begin{equation*}
\{E_{\be_1}^{m_1}...E_{\be_r}^{m_r}K_1^{l_1}...K_n^{l_n}F_{\be_r}^{k_r}...F_{\be_1}^{k_1}; \ m_i, k_i \in \N_0, 
\, l_i \in \Z \}
\end{equation*}
and writing $\bm =(m_1, \cdots, m_r)\in \N_0^r$, $\bk =(k_1,\cdots, k_r)\in \N_0^r$, $\bl =(l_1, \cdots, l_n)\in \Z^n$,
we abbreviate such a basis element as $E^\bm K^\bl F^\bk$. 

Next consider the adjoint action restricted to $U^0$. For $\ga\in Q$ we consider the root subspace 
\begin{equation*}
U_\ga = \{ X\in U \mid K_iXK_i^{-1} = q^{(\al_i,\ga)} X\},
\end{equation*}
and similarly defined 
$U^\pm_\ga = \{ X\in U^\pm \mid K_iXK_i^{-1} = q^{(\al_i,\ga)} X\}$.
Then we have
\begin{equation*}
U =\bigoplus_{\ga\in Q} U_\ga, \qquad 
U^+ =\bigoplus_{\ga\in Q^+} U^+_\ga, \qquad
U^- =\bigoplus_{\ga\in Q^+} U^-_{-\ga}.
\end{equation*}
Then $\dim U^+_\ga = \dim U^-_{-\ga} = K(\ga)$, where $K(\ga)$ is the Kostant partition function, i.e. 
the number of partitions of $\ga$ as a sum of positive roots. Note that $U_\be U_\ga\subset U_{\be+\ga}$ 
and $U^\pm_\be U^\pm_\ga\subset U^\pm_{\be+\ga}$. 
For $X\in U_\ga$ we say $\root(X)=\ga$, so for $X\in U_\ga$, $Y\in U_\be$ we have 
$\root(XY)=\root(X)+\root(Y)$. 

Finally, if we write $X\in U_\ga$ in the PBW-basis, $X=\sum_{\bm, \bk, \bl} 
\xi_{\bm, \bk, \bl} E^\bm K^\bl F^\bk$, then $\xi_{\bm, \bk, \bl}\not=0$ implies
$E^\bm K^\bl F^\bk \in U_\ga$. The PBW-basis is a joint eigenbasis for the adjoint action
of $U^0$. 


\section{The centralizer of the Cartan subalgebra}\label{sec:centralizerCSA}

In this section we study the structure of the $0$-root space of $U$ as well as some of its properties. 
So we study $U_0$, which is the centralizer of the Cartan subalgebra $U^0$ of the 
quantized enveloping algebra $U$. We are in particular interested in abelian subalgebras of $U_0$. 
These will be used later to define Mathieu modules.

We start by defining
\begin{equation}\label{eq:Ugasi}
U_{\ga, \si} = U^+_\si\, U^0\, U^-_{\ga - \si}, \qquad \ga \in Q, \si\in Q^+.
\end{equation}
Note that the space is trivial unless $\ga<\si$. Then 
the PBW-basis element $E^\bm K^\bl F^\bk \in U_{\ga,\si}$ if and only if
$\sum_i m_i\be_i =\si$ and $\sum_i k_i\be_i =\si-\ga$. As a consequence, we have
\begin{equation}\label{eq:Uga=sumUgasi}
U_\ga =\bigoplus_{\si\in Q^+} U_{\ga, \si}, \qquad 
U_0 =\bigoplus_{\si\in Q^+} U_{0, \si} = \bigoplus_{\si\in Q^+} U^+_\si U^0 U^-_{-\si}.
\end{equation}
Note that the PBW-basis gives a basis for the spaces $U_{\ga,\si}$. 
We extend the definition of $\root(X)=\ga$ for $X\in U_\ga$ to 
$\Eroot(X)=\si$ whenever $X\in U_{\ga,\si}$. In particular, the $\Eroot$ of a PBW-basis element
is well-defined. Similarly, the $\Froot(X)$ can be defined, but we do not use this. 

Recall that we have fixed a set $\Pi=\{\al_1,\cdots,\al_n\}$ of simple roots, and 
for $i\in \{1,\cdots,n\}$ we define the $i$-the height function
\begin{equation}\label{eq:defhi}
h_i\colon Q^+\to \N, \qquad h_i\Bigl( \sum_{j=1}^n m_j\al_j\Bigr) = m_i.
\end{equation}
For a PBW-basis element $E^\bm K^\bl F^\bk$ we define 
$h_i(E^\bm K^\bl F^\bk)= h_i(\Eroot(E^\bm K^\bl F^\bk))$.

\begin{defn}\label{def:hpmionU}
For $X=\sum_{\bm, \bk, \bl} \xi_{\bm, \bk, \bl} E^\bm K^\bl F^\bk \in U$ define
\begin{gather*}
h_i^-(X) = \min_{\xi_{\bm, \bk, \bl}\not=0} h_i(E^\bm K^\bl F^\bk), \qquad
h_i^+(X) = \max_{\xi_{\bm, \bk, \bl}\not=0} h_i(E^\bm K^\bl F^\bk).
\end{gather*}
\end{defn}

An alternative description of the height functions is 
the following. Let $X \in U$, then, upon decomposing $X$ in the PBW basis, we can group the PBW basis elements 
that have the same $i$-height obtaining  $X = \sum_{j = 0}^\infty X_j$, with $h_i(X_j) = j$.
Only a finite number of $X_j$ is nonzero, and 
$h_i^-(X) = \min_{X_j\not= 0} j$ and $h_i^+(X) = \max_{X_j\not= 0} j$.

We have that $h_i^+(X) = 0$ for all $i$ if and only if $X \in U^0$, but it is not true 
that $h_i^-(X) = 0$ for all $i$ implies $X \in U^0$.  
Furthermore, multiplying by elements of the Cartan subalgebra $U^0$ on the left or right 
does not alter the minimal or maximal $i$-heights; 
if $X \in U^0$ and $Y \in U$ then $h_i^\pm(XY) = h_i^\pm(YX) = h_i^\pm(Y)$. 

Finally, note that $h_i^-(X+Y) \geq \min (h_i^-(X),h_i^-(Y))$ for $X,Y\in U$.

\begin{lem}\label{lem:hiincreases}
Let $\bm, \bm', \bk, \bk' \in \N^r$. Then for every 
$X \in U^0U^-$, $Y \in U^+U^0$, we have 
$h_i^-(E^{\bm}E^{\bm'}X) = h_i^-(E^{\bm + \bm'}X)$ and $h_i^-(YF^\bk F^{\bk'}) = h_i^-(YF^{\bk + \bk'})$.
\end{lem}

\begin{proof}
Take $\ga = \Eroot(E^\bm) + \Eroot(E^{\bm'})$. Since $U^+$ is a subalgebra of $U$, 
we can decompose $E^\bm E^{\bm'} = \sum_{\bn \in N^r} \xi_\bn E^\bn \in U_{\ga}$ 
with respect to the PBW-basis. 
All of these elements in the PBW expansion satisfy $E^\bn \in U_{\ga, \ga}$. 
Since $E^{\bm + \bm'} \in U_{\ga, \ga}$ as well, we have 
$h_i^-(E^\bm E^{\bm'}X) = h_i(\ga) = h_i^-(E^{\bm + \bm'}X)$.
 
The proof of the other statement follows analogously. 
\end{proof}

We are in particular interested in the function $h_i^-$ on the centralizer algebra $U_0$.

\begin{prop}\label{prop:hi-onU0} 
For $X,Y\in U_0$ we have
$h_i^-(XY)   \geq \max (h_i^-(X), h_i^-(Y))$.
\end{prop}

\begin{proof} We start with $X$ and $Y$ elements from the PBW basis. 
For PBW-basis elements $E^\bm F^\bk$ and $E^{\bm'}F^{\bk'}$ elements in $U_0$,
we write $F^\bk E^{\bm'} = \sum \xi_{\bm'', \bl'', \bk''}E^{\bm''}K^{\bl''}F^{\bk''}$ in the PBW basis. 
Then
\begin{equation*}
\begin{split}
h_i^-(E^\bm F^\bk E^{\bm'}F^{\bk'}) & = h_i^-(\sum \xi_{\bm'', \bl'', \bk''} E^\bm E^{\bm''}K^{\bl''}F^{\bk''}F^{\bk'}) \\
& \geq \min_{\xi_{\bm'', \bl'', \bk''}\not=0} h_i^-(E^\bm E^{\bm''}K^{\bl''}F^{\bk''}F^{\bk'}) 
= \min_{\xi_{\bm'', \bl'', \bk''}\not=0} h_i^-(E^{\bm + \bm''}F^{\bk'' + \bk'}) 
\end{split}
\end{equation*}
using Lemma \ref{lem:hiincreases} and the fact that $h_i(E^{\bm}K^{\bl}F^{\bk})=h_i(E^{\bm}F^{\bk})$. 
Since $h_i^-(E^{\bm + \bm''}F^{\bk'' + \bk'})= h_i(\sum_{j=1}^r (m_j+m_j'')\be_j)
\geq h_i(\sum_{j=1}^r m_j\be_j) = h_i^-(E^{\bm}F^{\bk})$ it follows that 
$h_i^-(E^\bm F^\bk E^{\bm'}F^{\bk'}) \geq h_i^-(E^{\bm}F^{\bk})$.

Since $E^{\bm + \bm''}F^{\bk'' + \bk'}\in U_0$ we have $\sum_{j=1}^r (m_j+m_j'')\be_j = \sum_{j=1}^r (k_j''+k_j')\be_j$. Hence
\begin{equation*}
h_i^-(E^{\bm + \bm''}F^{\bk'' + \bk'})= h_i(\sum_{j=1}^r (k_j''+k_j')\be_j)
\geq h_i(\sum_{j=1}^r k'_j\be_j) = h_i(\sum_{j=1}^r m'_j\be_j) = h_i^-(E^{\bm'}F^{\bk'})
\end{equation*}
so that we have proved the statement for $X=E^{\bm}F^{\bk}$, $Y=E^{\bm'}F^{\bk'}$ in $U_0$. 

The proof of the case $X=E^{\bm}F^{\bk}$, $Y=\sum \xi_{\bm',\bl',\bk'} E^{\bm'}K^{\bl'}F^{\bk'}$ in $U_0$
uses that any non-trivial element $E^{\bm'}K^{\bl'}F^{\bk'}$ in the expansion for $Y$ is in $U_0$. 
Then we reduce to the previous case by  
\begin{multline*}
h_i^-(E^{\bm}F^{\bk}Y) = h_i^-(E^{\bm}F^{\bk}\sum \xi_{\bm',\bl',\bk'} E^{\bm'}K^{\bl'}F^{\bk'}) 
\geq \min_{\xi_{\bm',\bl',\bk'}\not=0} h_i^-(E^{\bm}F^{\bk}E^{\bm'}K^{\bl'}F^{\bk'}) \\
=\min_{\xi_{\bm',\bl',\bk'}\not=0} h_i^-(E^{\bm}F^{\bk}E^{\bm'}F^{\bk'}) 
\geq \max \Bigl( h_i^-(E^{\bm}F^{\bk}), \min_{\xi_{\bm',\bl',\bk'}\not=0} h_i^-(E^{\bm'}F^{\bk'})\Bigr) \\
= \max ( h_i^-(E^{\bm}F^{\bk}),h_i^-(Y))
\end{multline*}
using that the appearance of $K^{\bl'}$ is immaterial, and the value $h_i^-(E^{\bm}F^{\bk})$ is independent 
of the condition for the minimalization. 

The general case then follows by writing $X=\sum \xi_{\bm,\bl,\bk} E^{\bm}K^{\bl}F^{\bk}$ and follow
\begin{multline*}
h_i^-(XY) = h_i^-(\sum \xi_{\bm,\bl,\bk} E^{\bm}K^{\bl}F^{\bk}Y) 
\geq \min_{\xi_{\bm,\bl,\bk}\not=0} h_i^-(E^{\bm}K^{\bl}F^{\bk}Y) 
=\min_{\xi_{\bm,\bl,\bk}\not=0} h_i^-(E^{\bm}F^{\bk}Y) \\
\geq \max \Bigl( \min_{\xi_{\bm,\bl,\bk}\not=0} h_i^-(E^{\bm}F^{\bk}), h_i^-(Y) \Bigr) 
= \max ( h_i^-(X),h_i^-(Y))
\end{multline*}
again using that the appearance of$K^\bl$ is immaterial, and the value $h_i^-(Y)$ is independent of 
the condition for minimalization. 
\end{proof}

Consider a subset $S\subset \{1,\cdots, n\}$, then there exists an associated disjoint 
decomposition $Q^+ = Q^+_S \cap (Q_S^+)^c$; 
\begin{align*}
Q^+_S & = \{\ga \in Q^+| \ h_i(\ga) = 0\ \forall i \notin  S \} 
=  \{\ga = \sum_{j=1}^n b_j\al_j \in Q^+ \mid i \notin S \Rightarrow b_i = 0 \} \\
\end{align*}
are the roots that can be completely written in terms of the simple roots $\{\al_i\mid i\in S\}$.
Then 
\begin{align*}
(Q^+_S)^c & = \{\ga \in Q^+|  \ \exists i \notin S \colon \  h_i(\ga) > 0    \}  
= \{\ga = \sum_{j=1}^n b_j\al_j \in Q^+ \mid\ \exists i \notin S \colon \  b_i > 0\}.
\end{align*}
From \eqref{eq:Uga=sumUgasi} we get a decomposition for the centralizer algebra;
\begin{equation*}
U_0 = U_0^S \oplus I^S, \qquad  
U_0^S = \bigoplus_{\ga \in Q^+_S} U_{0, \ga}, \quad 
I^S = \bigoplus_{\ga \in (Q^+_S)^c} U_{0, \ga}.
\end{equation*}
Consider a PBW basis element $X \in U_0$, then $X\in I^S$ if and only if $h_i^-(X) > 0$ for some 
$i \notin S$ and $X \in U_0^S$ if and only if for all $i \notin S$ we have $h_i^+(X) = 0$. 

\begin{rmk}\label{rmk:UgS0}
Keeping in the Dynkin diagram of $\Lg$  only the vertices from $S$ and the corresponding edges, 
we obtain a Dynkin diagram to which we associate the Lie algebra $\Lg_S$. Restricting the diagonal 
matrix $D$ to the set $S$, we can similarly define the quantized universal enveloping algebra $U_q(\Lg_S)$.
Then $U_q(\Lg_S)\subset U_q(\Lg)$ is a Hopf subalgebra which is invariant for the adjoint action of 
$U^0=U^0(\Lg)$. Then $U_0^S$ is generated by $U_q(\Lg_S)_0=U_q(\Lg_S)\cap U_0$ and $K_i$ for $i\notin S$.  
\end{rmk}

Now we look at commutative subalgebras of the centralizer of the Cartan subalgebra. Recall 
the notational conventions for the Lie algebra $\Lg$, in particular its Cartan matrix 
$(a_{i,j})_{1\leq i,j\leq n}$ and the quantized universal enveloping 
algebra $U=U_q(\Lg)$ as in Definition \ref{def:QUE}. 

\begin{thm}\label{thm:commsubalgU0}
Let $S\subset \{1,\cdots,n\}$. 
Assume that $a_{i,j}=0$ is for each pair $(i,j)$ with $i\not=j$ and $i,j\in S$.  
Consider the corresponding decomposition $U_0 = U_0^S \oplus  I^S$, then 
$U_0^S$ is a commutative subalgebra of $U_0$ generated by 
\[
E_iF_i, \ i\in S, \qquad K^\pm_j, \ 1\leq j \leq n
\]
The subspace $I^S$ is a two-sided ideal of $U_0$. 
\end{thm}

\begin{rmk}\label{rmk:stronglyorthroots} 
(i) Recall that two non-proportional roots $\al,\be$ are strongly orthogonal if 
$\al\perp \be$ and if $\al\pm \be$ are not roots, which plays an important role
in determining maximal abelian subspaces in symmetric pairs, see \cite[VIII,\S 7]{Helg}
and \cite{Letz} for the quantum case. 
Note that the condition in 
Theorem \ref{thm:commsubalgU0} means that $\{ \al_i\mid i \in S\}$ forms a set of 
strongly orthogonal roots. Indeed, $\al_i-\al_j$ is not a root, and if $\al_i+\al_j$ 
would be a root, so would the reflection $\al_i-\al_j$ in the hyperplane orthogonal to $\al_j$.
See e.g. \cite{AgaoK} for classification results on maximal families of strongly orthogonal roots. 

\noindent
(ii) The case $S=\emptyset$ gives $U_0^\emptyset=U^0$ and  $I^\emptyset = \bigoplus_{\ga \in Q^+\setminus\{0\}} 
U_{0, \ga}$. Then $I^\emptyset$ is the kernel of the Harish-Chandra homomorphism \cite[\S6.3.4]{KlimS},  
and \cite[\S 7.4]{Dixm} for the classical case.
\end{rmk}

\begin{cor}\label{cor:thm:commsubalgU0} 
A $1$-dimensional representation of the commutative subalgebra $U_0^S$ extends
to a $1$-dimensional representation of $U_0$. 
\end{cor}

\begin{proof} Let $\pi\colon U_0^S\to \C$ be a $1$-dimensional representation, then 
we extend $\pi$ to $U_0$ by putting $\pi\vert_{I^S}=0$. Since $I^S$ is a $2$-sided 
ideal, $\pi\colon U_0\to \C$ is a representation. 
\end{proof}

\begin{proof}[Proof of Theorem \ref{thm:commsubalgU0}]
In this case the Lie algebra $\Lg_S$ as in Remark \ref{rmk:UgS0} consists of 
$|S|$ copies of $\mathfrak{sl}(2,\C)$, so the positive roots for $\Lg_S$ are 
just the simple roots (corresponding to $S$), i.e. $\Phi_S^+=\{\al_i\mid i\in S\}$. 
Note that $E_iF_i$, $i\in S$, and $K^\pm_j$, $1\leq j \leq n$ are in $U_0^S$. 
Also, since $a_{i,j}=0$ for $i\not= j$ and $i,j\in S$, it follows from the 
Serre relations of Definition \ref{def:QUE} that $E_iE_j=E_jE_i$, $F_iF_j=F_jF_i$ 
for all $i,j\in S$. Hence, $[E_iF_i, E_jF_j]=0$ for $i,j\in S$. And since $E_iF_i\in U_0$,
we see that $E_iF_i$, $i\in S$, and $K^\pm_j$, $1\leq j \leq n$, generate a 
commutative subalgebra $A$ of the subalgebra $U_0\subset U$ 
which only involves elements from the root lattice $Q_S^+$, hence $A\subset U_0^S$.  So it suffices to show $A\supset U_0^S$.

Now take a PBW basis element in $U_{0,\ga}\subset U_0^S$ for $\ga\in Q_S^+$, which can 
be written as $E_{i_1}^{k_1} \cdots E_{i_s}^{k_s} K^\bl F_{i_s}^{k_s} \cdots F_{i_1}^{k_1}$ 
where $S=\{i_1,\cdots, i_s\}$ since the positive roots of $\Lg_S$ are the simple roots 
$\al_{i_1},\cdots,\al_{i_s}$. The $E_{i_j}$'s, respectively $F_{i_j}$'s, commute 
amongst each other, and we can move the $K^\bl$ around at the cost of a power of $q$. 
So we can rewrite this element, up to a power of $q$, as 
$E_{i_1}^{k_1}F_{i_1}^{k_1} \cdots E_{i_s}^{k_s}  F_{i_s}^{k_s} K^\bl$.
It suffices to do the $U_q(\mathfrak{sl}(2))$-calculation that $E^kF^k$ is a polynomial
in $EF$ with coefficients polynomial in $K$, $K^{-1}$, see Lemma \ref{lem:sl2EnFn=p(EF)}.
This also shows that $U_0^S$ is an algebra, as for the $U_q(\Lsl(2,\C))$ calculations in 
Section \ref{sec:Mathieumodsl2C}.

To show that $I^S$ is an ideal, it suffices to take PBW-basis elements $E^\bm F^\bk \in U_0$
and $E^{\bm'}F^{\bk'}\in I^S$ and show that 
$E^{\bm'}F^{\bk'}E^\bm F^\bk\in I^S$ and $E^\bm F^\bk E^{\bm'}F^{\bk'}\in I^S$. 
Note we can assume $E^{\bm'}F^{\bk'}\in U_{0,\ga}$ with $\ga\in (Q^+_S)^c$. Pick 
$i\notin S$ with $h_i(\ga)>0$, then by Proposition \ref{prop:hi-onU0} we have 
\[
h_i^-(E^{\bm'}F^{\bk'}E^\bm F^\bk) \geq h_i^-(E^{\bm'}F^{\bk'}) = h_i(\ga) >0
\]
hence $E^{\bm'}F^{\bk'}E^\bm F^\bk\in I^S$. Similarly, the reversed order can be 
dealt with and obtain that $I^S$ is a two-sided ideal in $U_0$. 
\end{proof}


\section{Mathieu modules}\label{sec:Mathiemod}

We stick to the notation for the quantized enveloping algebra $U=U_q(\Lg)$, the corresponding
Cartan subalgebra $U^0$ and its centralizer $U_0$ in $U$. We view $U$ as a right $U_0$-module, and
recall that $U^0\subset U_0$. 

\begin{defn}\label{def:Mathieumod}
Let $V$ be any (left) $U_0$-module and consider the induced $U$-module $U \otimes_{U_0} V$. 
If $V=\bigoplus_\ga V_\ga$ is a weight module, i.e. decomposes in terms of 
finite-dimensional weight spaces for the $U^0$-action, we say 
that the induced module $M(V) = U \otimes_{U_0} V$ is a Mathieu module of $U$ induced by $V$.
We call $\dim V$ the rank of the Mathieu module $M$. The Mathieu module $M(V)$ is called degenerate in case 
$X\cdot v= 0$ for all $v\in V$ and all $X\in U_{0,\ga}$ for all $\ga \in Q^+\setminus\{0\}$. 
\end{defn}

Note that a weight module is a module with a direct sum decomposition with respect to the action of $U^0$. 
Here $\ga\colon U^0\to \C$ is a homomorphism, and then $V_\ga =\{v\in V \mid Kv=\ga(K)v\}$. For the 
adjoint action of $U^0$ on $U=U_q(\Lg)$ we obtain the decomposition in weight spaces $U_\la$, $\la\in Q$,
corresponding to the homomorphism $q^\la \colon U_0\to \C$, $K_i\mapsto q^{(\al_i,\la)}$. 

For a Mathieu module $M(V)$, the subspace $1 \otimes V \subset M(V)$  
is a  sub-$U_0$-module isomorphic to $V$. 
Observe that the Mathieu module is a weight module;
\begin{equation}\label{eq:Mathieuweightmod}
M(V)_\ga =\bigoplus_{\ga=q^\la \mu} U_\la \otimes V_\mu.
\end{equation}
Here $q^\la \mu\colon U^0\to \C$ defined by $q^\la \mu(K) = q^\la(K) \mu(K)$ for $K\in U^0$, since
all $K\in U^0$ are group-like elements.

Recall that the weight module $V=\bigoplus_{\ga} V_\ga$ is a highest, respectively lowest, weight module 
if the weights occurring are of the form $q^\la \mu$ for some fixed $\mu$ and $\la\in -Q^+$,
respectively $\la\in Q^+$. Assuming $V_\mu\not=\{0\}$, we say that $\mu$ is the highest, respectively lowest, weight of the $U$-module $V$. 

Note that the construction of Definition \ref{def:Mathieumod} is functorial, i.e. if
$\psi \colon V\to \tilde{V}$ is a $U_0$-module map between weight modules $V$ and $\tilde{V}$, then 
$M(\psi) = \Id\otimes \psi\colon M(V)\to M(\tilde{V})$ is a $U$-module morphism extending $\psi$, 
and using \eqref{eq:Mathieuweightmod} we find that $\psi$ is surjective, respectively injective,
if and only if $M(\psi)$ is surjective, respectively injective. So the Mathieu module 
is determined 
by the equivalence class of the $U_0$-module $V$.

\begin{lem}\label{lem:universalproperty}
Assume $W$ is $U$-module which is a weight module.
Let $V\subset W$ be a $U_0$-submodule,  and let $\tilde{V}$ 
be a $U_0$-module which is a weight module. 
Assume $\psi\colon \tilde{V} \to V$ is a $U_0$-module homomorphism, then 
there is a $U$-module homomorphism $\Psi\colon M(\tilde{V}) \to W$ extending $\psi$.
\end{lem}

\begin{proof} 
Consider the bilinear map 
\[
\Psi\colon U \times \tilde{V} \to W, \qquad \Phi(X, v) = X\cdot \psi(v) \in W, \qquad X \in U, \quad v \in \tilde{V}.
\]
Then for $Z\in U_0$ we have $\Psi(XZ, v) = \Psi(X, Z\cdot v)$, since $XZ\cdot \psi(v) = X\cdot \psi(Z\cdot v)$ 
as $\psi$ is a $U_0$-intertwiner. By universality we obtain a map, also denoted $\Psi \colon M(\tilde{V})
= U\otimes_{U_0} \tilde{V} \to W$, $\Psi(X\otimes v) = X\cdot \psi(v)$, 
which by construction intertwines the $U$-action. Moreover, $\Psi(1\otimes v) = \psi(v)$, so that 
$\Psi$ extends $\psi$. 
\end{proof}

\begin{prop}\label{prop:weightmodsarequotientsMathmod}
Let $W$ be an $U$-module generated by a weight vector $w\in W$, then 
$W$ is isomorphic to a quotient of a Mathieu module. In particular, an irreducible 
weight representation of $U$ is a isomorphic to a quotient of a Mathieu module.
\end{prop}

\begin{proof}
Define the $U_0$-module $V$ generated by $w$, i.e.  $V=U_0w$, then $V$ is a weight module 
with only one weight occurring, which is the same weight as that of $w$. 
The identity map is a $U_0$-module homomorphism $\io\colon V \to V \subset W$.
By Lemma \ref{lem:universalproperty}, there is a $U$-module homomorphism $\Psi\colon 
M(V) \to W$ extending $\io$. Note that $\Psi$ is surjective, since $w$ generates $W$. 
Hence $W \cong M (V)/\Ker(\Psi)$.   
\end{proof}

\begin{cor}\label{cor:prop:weightmodsarequotientsMathmod}
Let $W$ be an irreducible highest weight $U$-module, or an irreducible lowest weight $U$-module, 
then $W$ is isomorphic to a quotient of a Mathieu module of rank $1$. 
\end{cor}

We say that $W$ is an extremal weight $U$-module if 
$W$ is an irreducible highest weight $U$-module or an irreducible lowest weight $U$-module. 

\begin{proof} By Proposition \ref{prop:weightmodsarequotientsMathmod} it suffices to show that
we can take a Mathieu module of rank $1$. 
Let $w$ be the highest weight vector of $W$, then $U_0w=\C w$ since this is the only space
with the same weight as the weight of $w$. So  take the $U_0$-module
$\C w$ of dimension $1$ and apply the construction to obtain the corresponding Mathieu module $M$ 
of rank $1$. 
\end{proof}


\section{Mathieu modules of rank $1$}\label{sec:Mathieumodrank1}

\begin{lem}\label{lem:sizerank1Mathieumod}
Let $M(V)$ be a rank $1$ Mathieu module for $U$. Then, as vector spaces $U \cong M(V) \otimes U_0$.
\end{lem}

\begin{proof}
This follows from the associativity property of tensor products of modules over rings.
Let $V\cong\C$ as $U_0$-module, then 
\begin{gather*}
M(V) \otimes_{\C} U_0  \cong (U \otimes_{U_0} \C) \otimes_\C U_0 
\cong   U \otimes_{U_0} (\C \otimes_\C U_0) 
\cong U \otimes_{U_0} U_0  \cong U. \qedhere
\end{gather*}
\end{proof}

\begin{prop}\label{prop:r1M-maxpropersubmod}
Let $M(V)$ be a rank $1$ Mathieu module, then there exists a unique maximal proper 
submodule $W(V)$. So $M(V)/W(V)$ is the unique irreducible quotient of the Mathieu module.
\end{prop}

\begin{proof}
Let $V\cong \C_\la$ with weight $\la \colon U^0\to \C$, so that $M(V)\cong U\otimes_{U_0}\C$ has 
weight space decomposition $M(V)_{\la q^{\ga}}= U_\ga \otimes_{U_0} V$. 
In particular, for $\ga=0$, $M(V)_{\la}= U_0 \otimes_{U_0} V$ is one-dimensional.
Since a proper submodule $W$ is a weight module, $W$ cannot contain $M(V)_{\la}$ since 
$M(V)_{\la}=1\otimes V$ generates $M(V)$. So the union of all proper submodules is proper,
and gives the unique maximal proper submodule $W(V)$.
\end{proof}

In order to construct Mathieu modules of rank $1$ we consider Theorem \ref{thm:commsubalgU0}.
So take $S=\{i_1,\cdots, i_s\}$, $s=|S|$, as in Theorem \ref{thm:commsubalgU0}, and 
consider $\mu=(\mu_{i_1},\cdots, \mu_{i_s}) \in \C^s$,
$\mu_i\not=0$ for all $i\in S$
and $\la\colon U^0\to \C$. Define the one-dimensional module 
$\phi^S_{\la,\mu}\colon U_0=U_0^S\oplus I^S\to \C=\C^S_{\la,\mu}$ by 
\begin{equation*}
\Ker\, \phi^S_{\la,\mu} = I^S, \qquad  E_iF_i\mapsto \mu_i, \ i\in S, \qquad K_j\mapsto \la_j,\ 1\leq j\leq n. 
\end{equation*}
Note that allowing $\mu_i$'s to be zero would mean to consider a smaller subset of $S$. 

\begin{defn} 
Define $M^S_{\la,\mu}= M(\C^S_{\la,\mu})$ as the rank $1$ Mathieu modules induced 
by the one-dimensional $U_0$-representations $\phi^S_{\la,\mu}$. 
\end{defn}

In case $S=\emptyset$ we drop $\mu$ and $S$ from the notation. By the requirement that $\mu_i\not=0$ for 
all $i$, the Mathieu module $M^S_{\la,\mu}$ is degenerate if and only if $S=\emptyset$. 

Let $V(\la)= U_q(\Lg) \otimes_{U_q(\Lb^-)} \C_\la$ be the lowest weight Verma module,
where $\C_\la$ is the one-dimensional $U_q(\Lb^-)=U^0\otimes U^-$ module obtained by extending the 
one-dimensional $U^0$-representation $\la$ trivially to $U^-$. 
According to Proposition \ref{prop:weightmodsarequotientsMathmod} the module $V(\la)$ is 
a quotient of a Mathieu module.

\begin{prop}\label{prop:degMathieumodisVerma} 
Assume the lowest weight Verma module $V(\la)$ is irreducible, 
then $V(\la)\cong M_\la/W$ where $M_\la$ is the degenerate Mathieu module and $W$ its maximal 
proper invariant subspace.
\end{prop}

\begin{proof} Consider the invariant space $W_0$ of $M_\la$ generated by $F_i\otimes 1$, $1\leq i \leq n$. 
We first observe that $W_0$ is a proper subspace, and for this it suffices to show that 
$1\otimes 1\notin W_0$. Indeed, if it does then we have $X_i\in U$ so that 
$1\otimes 1 = \sum_{i=1}^n X_iF_i\otimes 1$, and decomposing $X_i = \sum_\be X_i^\be$ according
to $U = \bigoplus_{\be\in Q} U_\be$ we have 
$1\otimes 1 = \sum_{i=1}^n \sum_{\be} X_i^\be F_i\otimes 1$.
Considering the weight $\la$ we require that for non-zero terms in the sum we have 
$X_i^\be F_i\in U_0$, and then $h_i^-(X_i^\be F_i)>0$, so that $X_i^\be F_i\in I=I^\emptyset$ 
which acts as zero. So $1\otimes 1\notin W_0$.

Let $W$ be the maximal proper subspace, which contains $W_0$ by
Proposition \ref{prop:r1M-maxpropersubmod}. Then the image $v$ of $1\otimes 1$ 
in $M(\C_\la)/W$ satisfies $F_i\cdot v=0$ for all $1\leq i\leq n$. 
So it is a lowest weight vector of weight $\la$. Assuming $V(\la)$ is
irreducible, we find $M(\C_\la)/W \cong V(\la)$. 
\end{proof}
 
Next we discuss the unitarizability of the rank $1$ Mathieu modules.
We restrict to case of real $q$, and we consider the $\ast$-structures 
as in the classification of Twietmeyer \cite{Twie}, see \cite[\S 9.4]{CharP}. 
Then  the $\ast$-structure is given by an involutive Dynkin diagram
automorphism $\eta$ and a set of numbers $s_i\in \{\pm 1\}$, $1\leq i \leq n$,
so that 
\begin{equation}\label{eq:aststructures}
K_i^\ast = K_{\eta(i)}, \qquad E_i^\ast = s_i F_{\eta(i)} K_{\eta(i)},
\qquad F_i^\ast = s_i K^{-1}_{\eta(i)} E_{\eta(i)}
\end{equation}
with the condition that $s_i=1$ if $\eta(i)\not= i$. 
From \eqref{eq:aststructures} we see that $(U^0)^\ast=U^0$, and 
this gives $(U_\be)^\ast = U_{-\eta(\be)}$ extending $\eta$ to $Q$
by $\eta(\be) = \eta(\sum_{i=1}^n b_i\al_i)= \sum_{i=1}^n b_i\al_{\eta(i)}$. 

We extend $\phi^S_{\la,\mu} \colon U = \bigoplus_{\be\in Q} U_\be \to \C$ by
first projecting on $U_0$ and next applying the $1$-dimensional representation
$\phi^S_{\la,\mu}$ of $U_0$.

\begin{prop}\label{prop:unitaryMathieumod}
Let the $\ast$-structure be given by \eqref{eq:aststructures}, and assume 
$\phi^S_{\la,\mu} \colon U\to \C$ as defined above is a positive linear functional.
Then $M(\C^S_{\la,\mu})/N$ is an irreducible unitary $U$-module, where 
\[
N= \{ X\cdot (1\otimes 1)\mid \phi^S_{\la,\mu}( X^\ast X)=0 \}.
\]
\end{prop}

Note that $X\in U_\be$  gives $X^\ast X \in U_{\be- \eta(\be)}$ so that 
$U_\be\subset N$  in case $\eta(\be)\not=\be$. 

\begin{proof} 
Since $S$, $\la$ and $\mu$ are fixed, we use the notation $\phi=\phi^S_{\la,\mu}$ in the proof. 
Note that for $X\in U$, $Z\in U_0$ we have $\phi(XZ)=\phi(ZX)=\phi(Z)\phi(X)$, since this 
is true for $X\in U_\be$ for any $\be\in Q$ by $U_0U_\be\subset U_\be$ and $U_0$ being $\ast$-invariant. 
Define the sesquilinear form 
\[
\langle \cdot, \cdot \rangle \colon M(\C^S_{\la,\mu}) \times M(\C^S_{\la,\mu}) \to \C, 
\qquad \langle X\cdot(1\otimes 1), Y\cdot(1\otimes 1) \rangle = \phi(Y^\ast X),
\]
which is well-defined by the previous observation. Then the 
Cauchy-Schwarz inequality 
\[
|\phi(Y^\ast X)|^2 \leq \phi(X^\ast X) \phi(Y^\ast Y)
\]
implies that $N$ is invariant subspace. The space $M(\C^S_{\la,\mu})/N$
is an inner product space and the action of $U$ is unitary by construction. 

The subspace $V$ generated by the action of $U$ on the image of $1\otimes 1$ in $M(\C^S_{\la,\mu})/N$
is an invariant subspace. Since the representation is unitary, we know that the orthocomplement is
invariant as well and we show it is trivial.
So assume $X\cdot (1\otimes 1 + N)$ is perpendicular to $V$, then 
\[
\phi(Y^\ast X) = \langle X\cdot (1\otimes 1 + N), Y\cdot (1\otimes 1 + N) \rangle = 0 \quad \forall \,
Y \in U. 
\]
In particular, taking $Y=X$ gives $\phi(X^\ast X)=0$ and $X\cdot (1\otimes 1)\in N$, so the 
orthocomplement is trivial. 
\end{proof}

Since we require $\phi^S_{\la,\mu}$ to be a positive functional, we see that 
we require $\overline{\la_i} =\la_{\eta(i)}$ and $\overline{\mu_i} = \mu_{\eta(i)}$,
since $(E_iF_i)^\ast= E_{\eta(i)}F_{\eta(i)}$ and $K_i^\ast=K_{\eta(i)}$. 
Assuming that $S\subset \{ i\mid \eta(i)=i\}$, we have $\mu_i, \la_i\in \R$ and 
\[
E_iF_i = s_i K_i F_i^\ast F_i \quad \Longrightarrow \quad 
\mu_i =s_i \la_i \phi^S_{\la,\mu}(F_i^\ast F_i)
\]
so that $\mu_i\la_i>0$ in case $s_i=1$ and $\mu_i\la_i<0$ in case $s_i=-1$. 


\section{Mathieu Modules for $U_q(\Lsl(2,\C))$}\label{sec:Mathieumodsl2C}

In this section we discuss Mathieu modules for the simplest quantum algebra $U_q(\Lsl(2,\C))$. 
The quantum algebra $U_q(\Lsl(2,\C))$ is of type $A_1$ and has the $1\times1$ Cartan matrix $(2)$. 
By Definition \ref{def:QUE}
$U_q(\Lsl(2,\C))$ is generated by elements $E=E_1$, $F=F_1$, $K=K_1$, where the quantum Serre relations are void. 
The root system is $\Phi=\{\pm \al\}$. 

We show that the Mathieu modules can be used to obtain all irreducible unitary modules for the 
$U_q(\Lsu(1,1))$, i.e. the quantum algebra $U_q(\Lsl(2,\C))$ equipped with the $\ast$-structure 
\begin{equation}\label{eq:starforUsu11}
K^\ast = K, \quad E^\ast = -FK, \quad F^\ast=-K^{-1}E, 
\end{equation}
see \eqref{eq:aststructures}. 

\begin{defn}\label{def:admissibletypeI}
A $U_q(\Lsu(1,1))$-module $V$ is admissible if $V$ has a weight space decomposition 
$V = \bigoplus V_\sigma$ for the action of $K$ with finite-dimensional weight spaces 
$V_\sigma$. The module $V$ is of type I if the eigenvalues $\sigma$ are of the form 
$q^\tau$ for $\tau\in \R$. 
\end{defn}

The unitary admissible type I representations of $U_q(\Lsu(1,1))$ have been 
classified by Vaksman and Korogodski\u\i\ \cite{VaksK}, Burban and Klimyk \cite{BurbK} and Masuda et al. \cite{MasuMNNSU}, and
they play an important role in the harmonic analysis on the quantum group analog of 
$SU(1,1)$. The purpose is to show that one can obtain these representations 
from the Mathieu modules for $U_q(\Lsl(2,\C))$. 

\subsection{Mathieu modules for $U_q(\Lsl(2,\C))$}\label{ssec:Mathieumodsl2C}

For future reference we collect some well-known commutation relations 
in Lemma \ref{lem:commutationsl2}. The proof is a straightforward verification by induction and the 
relations of Definition \ref{def:QUE} for the case $U_q(\Lsl(2,\C))$, see e.g. \cite{KlimS}.  

\begin{lem} \label{lem:commutationsl2}
For $n, m \in\N_0$ we have 
\begin{enumerate}[(i)]
\item $K^nE^m = q^{2mn}E^mK^n$ and $K^nF^m = q^{-2mn}F^mK^n$, \\[1mm]
\item $\displaystyle{EF^n = F^nE + \frac{q^n-q^{-n}}{q-q^{-1}}F^{n-1}
\frac{q^{1 - n}K - q^{n - 1}K^{-1}}{q-q^{-1}}}$,\\[1mm]
\item $\displaystyle{FE^n = E^nF - + \frac{q^n-q^{-n}}{q-q^{-1}}E^{n-1}
\frac{q^{n-1}K - q^{1-n}K^{-1}}{q-q^{-1}}}$.
\end{enumerate}
\end{lem}

The PBW basis is now given by $E^mK^lF^k$ for $m,k\in\N_0$, $l\in\Z$. 
For a given element $E^mK^lF^k$ of the PBW basis we have $KE^mK^lF^k = q^{2m - 2k}E^mK^lF^kK$, so that 
$E^mK^lF^k \in U_0$ if and only if $m = k$.  
Let $X = \sum_{m, l, k} \xi_{m, l, k} E^mK^lF^k$ (finite sum) be an arbitrary element of $U_0$ written in its PBW basis decomposition. Then each $E^mK^lF^k$ is also in $U_0$ and so  
$\xi_{m, l, k} \neq 0$ implies $k = m$. In this case, the element $E^mK^lF^m \in U_{0, m\alpha}$. 
This proves Lemma \ref{lem:sl2U0}.

\begin{lem}\label{lem:sl2U0}
For each $n \in \N_0$, $U_{0, n\alpha}$ is one-dimensional $U^0$-module spanned by $E^nF^n$.
Moreover, $\{E^nK^lF^n; \ n \in \N_0, l \in \Z  \}$
is a basis for $U_0$.
\end{lem}

\begin{lem}\label{lem:sl2EnFn=p(EF)}
$E^nF^n= (EF)^n + \sum_{i=0}^{n-1} (EF)^i c_i(K,K^{-1})$ for some polynomial $c_i(K,K^{-1})$ 
in $K$ and $K^{-1}$ for all $n\in\N$.
\end{lem}

\begin{proof} The case $n=1$ is trivial, and the induction step follows from
Lemma \ref{lem:commutationsl2}(ii);
\begin{gather*}
E^{n+1}F^{n+1} = E^n EF^{n}F = E^nF^n EF + \frac{q^{n}-q^{-n}}{q-q^{-1}} E^nF^{n-1} 
\frac{q^{1-n}K-q^{n-1}K^{-1}}{q-q^{-1}}F.
\end{gather*}
Moving $F$ through $K^{\pm 1}$, we can apply the induction hypothesis. Since $EF$ commutes with $K^{\pm 1}$, the result follows.
\end{proof}

\begin{cor}\label{cor:lem:sl2EnFn=p(EF)}
$U^0=\C[EF, K, K^{-1}]$ is a commutative algebra.
\end{cor}

Note that Corollary \ref{cor:lem:sl2EnFn=p(EF)} is the special case $S=\{1\}$ in 
the notation of Theorem \ref{thm:commsubalgU0}. 

As in Section \ref{sec:Mathieumodrank1} 
we define the $1$-dimensional $U_0$-modules $\C_{\la, \mu} \cong \C$ by 
choosing $K\cdot1 = \la 1$ and $EF\cdot1 = \mu 1$, where $\la, \mu \in \C$, $\la \neq 0$.
The case $\mu=0$ corresponds to the degenerate case. 
Denote this $1$-dimensional $U_0$-representation by $\phi=\phi_{\la,\mu}$.  
We then consider the Mathieu module $M(\C_{\la, \mu}) = U_q(\Lsl(2,\C)) \otimes_{U_0} \C_{\la, \mu}$ associated to this 
$1$-dimensional $U_0$-module. We denote $1$ for the element $1\otimes 1\in M(\C_{\la, \mu})$.

\begin{prop}\label{prop:basisMatmodsl2}
The set $\{E^n\cdot1\}_{n\in\N}\cup \{1\} \cup \{ F^n\cdot1\}_{n\in \N}$ is a basis of the $U_q(\Lsl(2,\C))$-module  
$M(\C_{\la, \mu})$ and the generators  
act on elements of this basis as follows:
\begin{enumerate}[(i)]
\item $K (E^n\cdot1) = q^{2n}\la E^n\cdot1$ and  $K (F^n\cdot1) = q^{-2n}\la F^n\cdot1$ for $n\in \N_0$, 
\item $E (E^n\cdot1) = E^{n+1}\cdot1$ and $F (F^n \cdot1) = F^{n + 1}\cdot 1$ for $n\in\N_0$, 
\item $\displaystyle{E (F^n\cdot1) = \bigl( 
\mu + \frac{(q^{n-1}-q^{1-n})(q^{-n}\la-q^n\la^{-1})}{(q-q^{-1})^2}\bigr) F^{n - 1}\cdot1}$ for $n \in \N$,
\item $\displaystyle{F (E^n \cdot1) = 
\bigl( \mu - \frac{(q^{n}-q^{-n})(q^{n-1}\la-q^{1-n}\la^{-1})}{(q-q^{-1})^2}\bigr)
E^{n - 1}\cdot 1}$ for $n\in \N$.
\end{enumerate}
\end{prop}

\begin{proof}
The elements $E^n\cdot 1$, $1$ and $F^n\cdot 1$ are non-zero by Lemma 
\ref{lem:sizerank1Mathieumod}
and they are linearly independent as weight vectors for  different weights. 
To show that they span the module we consider first the case where $E^mK^lF^k$ where $m \geq k$. 
We write $E^mK^lF^k \cdot 1 = E^{m - k}E^kK^lF^k\cdot 1 = \phi(E^kK^lF^k)E^{m - k}\cdot1$. 
For $k \geq m$ the situation is slightly more complicated. 
Write $E^mK^lF^k = q^{2l(m-k)} E^mF^{k-m} K^lF^m$ and next 
use Lemma  \ref{lem:commutationsl2}(ii) repeatedly to find that $E^mK^lF^k = F^{k - m}Z$ for some $Z \in U_0$. 
Hence $E^mK^lF^k\cdot1 = \phi(Z)F^{k - m}\cdot 1$.

The action of the generators on these elements in (i) and (ii) follow.  
For (iii) we have by Lemma \ref{lem:commutationsl2},
\begin{gather*}
EF^n\cdot 1 = F^{n-1}FE\cdot 1 + \frac{(q^n-q^{-n})(q^{1-n}\la - q^{n-1}\la^{-1})}{(q-q^{-1})^2}
 F^{n-1}\cdot 1,
\end{gather*}
and using $FE\cdot 1 = EF \cdot 1 - (q-q^{-1})^{-1}(K-K^{-1})\cdot 1 = \bigl(\mu + (q-q^{-1})^{-1}(\la -\la^{-1}\bigr)\, 1$ 
we find (iii) after a straightforward calculation. The proof of (iv) is similar and slightly simpler.
\end{proof}

From Proposition \ref{prop:basisMatmodsl2} we see that the representation space has a weight space decomposition for 
the action of $K$; $M(\C_{\la, \mu}) = \bigoplus_{k\in\Z} M(\C_{\la, \mu})_{\la q^{2k}}$,
where each $M(\C_{\la, \mu})_{\la q^{2k}}$ is $1$-dimensional and spanned by $E^k\cdot 1$ if $k>0$,
by $F^k\cdot 1$ if $k<0$ and by $1$ if $k=0$. Here we use $\la q^{2k}\colon U_0\to \C$ as the 
homomorphism sending $K\mapsto \la q^{2k}$, which corresponds to $\la q^{k\al}$. 
Proposition \ref{prop:basisMatmodsl2} shows that
\begin{equation}\label{eq:EFonweightspacesMathieusl2}
E \colon M(\C_{\la, \mu})_{\la q^{2k}} \to M(\C_{\la, \mu})_{\la q^{2(k+1)}}, \qquad
F \colon M(\C_{\la, \mu})_{\la q^{2k}} \to M(\C_{\la, \mu})_{\la q^{2(k-1)}}. 
\end{equation}

Recall the Casimir element
\begin{equation}
\Om = EF + \frac{Kq^{-1} + K^{-1} q}{(q - q^{-1})^2} =
 FE + \frac{Kq + K^{-1} q^{-1}}{(q - q^{-1})^2}
\end{equation}
for $U_q(\Lsl(2,\C))$, see e.g. \cite[\S 3.1.1]{KlimS}. Then 
$\Om$ is central, and it generates the center of $U_q(\Lsl(2,\C))$. 
Using Proposition \ref{prop:basisMatmodsl2} we can calculate the action of $\Om$ on 
any basis element of the $U_q(\Lsl(2,\C))$-module  $M(\C_{\la, \mu})$.

\begin{cor}\label{cor:lem:basisMatmodsl2}
The Casimir operator $\Om$ acts as the constant $\mu + \frac{q^{-1} \la  + q \la^{-1}}{(q - q^{-1})^2}$ 
times the identity 
on the $U_q(\Lsl(2,\C))$-module  
$M(\C_{\la, \mu})$. 
\end{cor}

\subsection{Reducibility}\label{ssec:reducibility}

The Mathieu module $M(\C_{\la, \mu})$ is admissible in the sense of 
Definition \ref{def:admissibletypeI}, since it has a weight space decomposition 
with finite-dimensional weight spaces.
Hence, in case $M(\C_{\la, \mu})$ is reducible, a non-trivial invariant subspace has a weight space decomposition.
Since the weight spaces are $1$-dimensional, we can only have a non-trivial invariant subspace in 
case $E$, respectively $F$, kills a weight space. 
From Proposition \ref{prop:basisMatmodsl2} we see that this can only happen in cases (iii) and (iv). 

In case (iii), $E$ kills a weight space if there exists  $n_E\in\N$ with 
\begin{equation}\label{eq:E-reducibility}
(q^{n_E-1}-q^{n_E-1})(q^{-n_E}\la-q^{n_E}\la^{-1})
+ (q - q^{-1})^2\mu = 0,
\end{equation}
and then $E\cdot(F^{n_E}\cdot1)=0$. 
Note that for fixed $\la$ and $\mu$, at most one solution $n_E\in \N$ for \eqref{eq:E-reducibility} exists. 
In this case the submodule $M^-_{n_E} = \bigoplus_{k\leq - n_E} M(\C_{\la, \mu})_{\la q^{2k}}$,
being the span of $F^{n_E+p}\cdot 1$, $p\in \N_0$,  is invariant. 
The spectrum of $K$ on the invariant subspace is $\la q^{-2n_E-2\N_0}$, so that 
we can consider $M^-_{n_E}$ as a highest weight representation. 

Similarly, in case (iv), we only get a zero action by $F$ on $E^n\cdot 1$ if there exists $n_F\in \N$ so that 
\begin{equation}
\label{eq:F-reducibility}
(q^{n_F}-q^{-n_F})(q^{n_F-1}\la-q^{1-n_F}\la^{-1})
- (q - q^{-1})^2\mu = 0,
\end{equation}
so that $F\cdot(E^{n_F}\cdot1)=0$. 
Again, there is at most one solution of \eqref{eq:F-reducibility} in $\N$.   
The submodule $M^+_{n_F} = \bigoplus_{k\geq n_F} M(\C_{\la, \mu})_{\la q^{2k}}$,
being the span of $E^{n_F+p}\cdot 1$, $p\in \N_0$,  is invariant. 
The spectrum of $K$ on the invariant subspace is $\la q^{2n_F+2\N_0}$, so that 
we can consider $M^+_{n_E}$ as a lowest weight representation. 

These considerations prove the first part of Proposition \ref{prop:reducibility}. 

\begin{prop}\label{prop:reducibility} The Mathieu module $M(\C_{\la, \mu})$ 
is generically irreducible. More precisely, assume that 
\eqref{eq:F-reducibility} has no solution $n_F\in \N$ and that 
\eqref{eq:E-reducibility} has no solution $n_E\in \N$, then $M(\C_{\la, \mu})$ 
is irreducible. 
Conversely, if $M(\C_{\la, \mu})$ is irreducible, then 
\eqref{eq:F-reducibility} has no solution $n_F\in \N$ and  
\eqref{eq:E-reducibility} has no solution $n_E\in \N$. 
\end{prop}

\begin{proof}
It remains to prove the converse statement. Since $M(\C_{\la, \mu})$ is the 
sum of the weight spaces, and, using the PBW basis, the only elements in 
$U_q(\Lsl(2,\C))$ mapping $M(\C_{\la, \mu})_{\la}$ to 
$M(\C_{\la, \mu})_{\la q^{2k}}$ ($k\geq 0$) are elements from 
$E^k U_0$. By irreducibility, the map $E^k$ has to be non-zero, and by 
\eqref{eq:EFonweightspacesMathieusl2}, we see that each 
$E \colon M(\C_{\la, \mu})_{\la q^{2p}} \to M(\C_{\la, \mu})_{\la q^{2(p+1)}}$
for $0\leq p < k$ has to be non-zero. Since $k$ is arbitrary, we find that 
\eqref{eq:E-reducibility} has no solution $n_E\in \N$.

The statement for \eqref{eq:F-reducibility} is proved similarly. 
\end{proof}

In case there exists a $n_E\in \N$ satisfying \eqref{eq:E-reducibility} and there 
exists no $n_F\in \N$ satisfying \eqref{eq:F-reducibility}, the quotient
$M(\C_{\la, \mu})/M^-_{n_E}$ gives an irreducible $U_q(\Lsl(2,\C))$-representation,
which we can view as a lowest weight module with lowest weight 
$\la q^{2-2n_E}$. Similarly, 
in case there exists a $n_F\in \N$ satisfying \eqref{eq:F-reducibility} and there 
exists no $n_F\in \N$ satisfying \eqref{eq:E-reducibility}, the quotient
$M(\C_{\la, \mu})/M^+_{n_F}$ gives an irreducible $U_q(\Lsl(2,\C))$-representation,
which we can view as a highest weight module with highest weight 
$\la q^{2n_F-2}$.
In case there exists a solution $n_E\in \N$ to \eqref{eq:E-reducibility}
and a solution $n_F\in \N$ to \eqref{eq:F-reducibility}, then the 
$M(\C_{\la, \mu})/\bigl(M^-_{n_E} + M^+_{n_F}\bigr)$ is a finite-dimensional 
irreducible $U_q(\Lsl(2,\C))$-representation.

\subsection{Equivalence}\label{ssec:equivalence}

In general the equivalence question for general Mathieu modules seems to be difficult.
For the case of $U_q(\Lsl(2,\C))$ and irreducible Mathieu modules, it is possible to 
describe it in detail.

\begin{prop}\label{prop:equivalenceirrMathieumodulessl2}
Assume that $M(\C_{\la, \mu})$ and $M(\C_{\la', \mu'})$ are irreducible Mathieu modules. 
Then $M(\C_{\la, \mu}) \cong M(\C_{\la', \mu'})$ if and only if 
there exists $n\in \Z$ with $\la'=\la q^{2n}$ and 
\[
\mu'=\mu -  \frac{(q^n-q^{-n})(\la q^{n-1}-\la^{-1}q^{1-n})}{(q-q^{-1})^2}.
\]
\end{prop}

\begin{proof} 
Assume first that the modules are equivalent. Since the spectrum of $K$ in both modules has to be equal, we find
$\la q^{2\Z}= \la' q^{2\Z}$. Hence, there exists $n\in\Z$ with $\la'=\la q^{2n}$. 
By considering the action of the Casimir element $\Omega$, Corollary \ref{cor:lem:basisMatmodsl2} gives the 
relation between $\mu$ and $\mu'$.

To prove the converse, we use Lemma \ref{lem:universalproperty}. Let $W=M(\C_{\la', \mu'})$. 
Let $\tilde{V}= \C_{(\la,\mu)} \cong \C\cdot 1_{(\la,\mu)} \subset 
M(\C_{\la, \mu})$, stressing the dependence on $(\la,\mu)$. Then we define 
\begin{equation*}
\begin{split}
&\psi \colon   \C\cdot 1_{(\la,\mu)} \to \C F^n\cdot 1_{(\la',\mu')}, 
\quad 1_{(\la,\mu)} \mapsto  F^n\cdot 1_{(\la',\mu')}, 
\qquad n\in \N_0, \\
&\psi \colon   \C\cdot 1_{(\la,\mu)} \to \C E^{-n}\cdot 1_{(\la',\mu')}, 
\quad 1_{(\la,\mu)} \mapsto  E^{-n}\cdot 1_{(\la',\mu')}, 
\qquad -n\in \N. \\
\end{split}
\end{equation*}
By a straightforward calculation using Proposition \ref{prop:basisMatmodsl2} we see that $\psi$ intertwines the 
action of $K$ and $EF$. Then $V$, the image of $\psi$, is a $U_0$-submodule of $M(\C_{\la, \mu})$.
Lemma \ref{lem:universalproperty} gives an intertwiner 
$\Psi\colon M(\tilde{V})=M(\C_{\la, \mu}) \to W=M(\C_{\la', \mu'})$,
which is non-zero, since it extends the non-zero map $\psi$. Since 
$M(\C_{\la, \mu})$ and $M(\C_{\la', \mu'})$ are irreducible, they are equivalent. 
\end{proof}

\begin{rmk}\label{rmk:prop:traceequivalenceirrMathieumodulessl2}
Introduce the map $\Tr(M) \colon (U^0)^\ast \times U_0 \to \C$ for a weight module $M=\bigoplus_\la M_\la$ 
by $\Tr(M)(\la,X)= \Tr(X\vert_{M_\la})$, see Mathieu \cite[\S 2]{Math}. 
Using Proposition \ref{prop:basisMatmodsl2} we obtain 
\begin{gather*}
\Tr(M(\C_{\la, \mu}))(\la q^{2k}, (EF)^jK^l) =
(\la q^{2k})^l \left( \mu - \frac{(q^k-q^{-k})(q^{k-1}\la-q^{1-k}\la^{-1})}{(q-q^{-1})^2}\right)^j
\end{gather*}
for $j\in \N_0$, $l\in \Z$, $k\in\Z$. By the explicit expression we see that 
$\Tr(M(\C_{\la, \mu}))=\Tr(M(\C_{\la', \mu'}))$ if and only if $(\la,\mu)$ and $(\la',\mu')$ are
related as in Proposition \ref{prop:equivalenceirrMathieumodulessl2}. 
\end{rmk}

\begin{rmk}\label{rmk:prop:equivalenceirrMathieumodulessl2}
The proof of Proposition \ref{prop:equivalenceirrMathieumodulessl2} is in case of irreducible Mathieu modules.
It is straightforward to write down the intertwiner explicitly. E.g. in case $n\in \N_0$ we have
\begin{gather*}
\Psi(E^k\cdot 1_{(\la,\mu)})= (E^kF^n)\cdot 1_{(\la',\mu')}, \qquad
\Psi(F^k\cdot 1_{(\la,\mu)})= F^{k+n}\cdot 1_{(\la',\mu')}.
\end{gather*}
Assuming $(\la,\mu)$ and $(\la',\mu')$ related as in Proposition \ref{prop:equivalenceirrMathieumodulessl2}
we can check that $\Psi$ intertwines the action using 
Proposition \ref{prop:basisMatmodsl2} directly. There are two non-trivial relations to check, namely
$\Psi(E(F^k\cdot 1_{(\la,\mu)}))=E\cdot \Psi(F^k\cdot 1_{(\la,\mu)})$ and
$\Psi(F(E^k\cdot 1_{(\la,\mu)}))=F\cdot \Psi(E^k\cdot 1_{(\la,\mu)})$. 
In the first case, Proposition \ref{prop:basisMatmodsl2} and the relation between $(\la,\mu)$ and $(\la',\mu')$
give the result. In the second case, the left hand side follows from Proposition 
\ref{prop:basisMatmodsl2}. For the right hand side we use Lemma \ref{lem:commutationsl2}(ii), (iii)
to write 
\begin{gather*}
FE^kF^n = E^{k-1}EF^{n}F-\frac{q^k-q^{-k}}{q-q^{-1}}E^{k-1}\frac{q^{k-1}K-q^{1-k}K^{-1}}{q-q^{-1}}F^n,  \\
 E^{k-1}EF^{n}F =  E^{k-1}F^{n}(EF) +
 E^{k-1} \frac{q^n-q^{-n}}{q-q^{-1}} F^{n-1} \frac{q^{1-n}K -q^{n-1}K^{-1}}{q-q^{-1}}F.
\end{gather*}
Then $F\cdot \Psi(E^k\cdot 1_{(\la,\mu)})= FE^kF^n\cdot 1_{(\la',\mu')}$ can be calculated directly in terms of $(\la',\mu')$. Using the relation between $(\la,\mu)$ and $(\la',\mu')$ then shows equality 
with the left hand side. Similarly, we have an explicit intertwiner for $-n\in \N$. 

Now the transition $(\la,\mu)\mapsto (\la',\mu')$ is invertible, and of the same type, i.e. 
$\la=\la' q^{-2n}$ and 
\[
\mu=\mu' -  \frac{(q^{-n}-q^{n})(\la' q^{-n-1}-(\la')^{-1}q^{1+n})}{(q-q^{-1})^2}.
\]
So, we then similarly find an intertwiner $\Psi'\colon M(\C_{\la',\mu'})\to M(\C_{\la,\mu})$. By 
considering the action on each of the basis vectors, we can obtain $\Psi'\circ \Psi = \phi_{\la,\mu}(F^nE^n) \textrm{Id}$. We will not use this result, and we skip its proof.
\end{rmk}

\subsection{Unitarizability}\label{ssec:unitarizability}

Next we consider which Mathieu modules for $U_q(\Lsl(2,\C))$ can be made into 
unitary representations for the $\ast$-structure \eqref{eq:starforUsu11} 
corresponding to the quantized universal enveloping algebra $U_q(\Lsu(1, 1))$. 
Recall that we assume $0<q<1$.

Observe that $K^\ast = K$ and $EF = -EE^\ast K^{-1}$, so that acting on $1\in M(\C_{\la, \mu})$
and recalling that $EE^\ast$ is a positive operator, we find the necessary conditions
\begin{equation}\label{eq:unitary-necessarycond}
\la \in\R\setminus\{0\}, \qquad \mu\la < 0,
\end{equation}
for $M(\C_{\la, \mu})$ to be unitary. 
In particular, for type I representations, see Definition \ref{def:admissibletypeI},  
we need $\la>0$ and $\mu < 0$,
see Section \ref{sec:Mathieumodrank1} and Proposition \ref{prop:unitaryMathieumod}. 

Since the basis vectors of Proposition \ref{prop:basisMatmodsl2} are eigenvectors of the 
(formally) self-adjoint operator $K$ for different eigenvalues, this constitutes 
an orthogonal basis in case $M(\C_{\la, \mu})$ is unitary. 
See Schm\"udgen \cite[Ch.~8]{Schm} for the notion of a representation by unbounded operators, where in 
this case the common domain is the finite linear combinations of the basis vectors of 
Proposition \ref{prop:basisMatmodsl2}. 

Assume $M(\C_{\la, \mu})$ is a unitary module for $U_q(\Lsu(1,1))$ with respect to 
the inner product $\langle \cdot | \cdot \rangle$. 
We derive a recursive expression; take $n>0$ and 
\begin{gather*}
\langle E^n \cdot 1 | E^n \cdot 1\rangle = \langle E^\ast E^n \cdot 1 | E^{n-1} \cdot 1\rangle =
- \langle F K E^n \cdot 1 | E^{n-1} \cdot 1\rangle = \\ 
- q^{2n}\la  
\bigl(
\mu - \frac{(q^{n}-q^{-n})(q^{n-1}\la-q^{1-n}\la^{-1})}{(q-q^{-1})^2}
\bigr) 
\langle E^{n-1} \cdot 1 | E^{n-1} \cdot 1\rangle 
\end{gather*}
using Proposition \ref{prop:basisMatmodsl2}. 
This is a simple recursion, and we find, using $\mu =M/(q - q^{-1})^2$ 
and normalizing $\langle 1 | 1\rangle=1$, 
\begin{equation*}
\langle E^n \cdot 1 | E^n \cdot 1\rangle 
= \left( \frac{q}{(q-q^{-1})^2}\right)^n \prod_{k=0}^{n-1}   
\bigl(1 - (\la^2+ q^2 + qM\la) q^{2k} + \la^2q^2 q^{4k}\bigr)
\end{equation*}
so that $\langle E^n \cdot 1 | E^n \cdot 1\rangle>0$ for all $n\in\N$ if and only if 
$1 - (\la^2+ q^2 + qM\la)x + \la^2 q^2x^2>0$ for 
all $x\in q^{2\N_0}$. 

Similarly, we find 
\begin{gather*}
\langle F^n \cdot 1 | F^n \cdot 1\rangle = 
- \langle K^{-1} E F^n \cdot 1 | F^{n-1} \cdot 1\rangle = \\ 
- \la^{-1} q^{2(n-1) }
\bigl(
\mu + \frac{(q^{n-1}-q^{1-n})(q^{-n}\la-q^n\la^{-1})}{(q-q^{-1})^2}
\bigr)
\langle F^{n-1} \cdot 1 | F^{n-1} \cdot 1\rangle. 
\end{gather*}
and hence
\begin{equation*}
\langle F^n \cdot 1 | F^n \cdot 1\rangle =  \left( \frac{q^{-1}}{(q-q^{-1})^2} \right)^n \prod_{k=0}^{n-1}
\bigl( 1 - (q^2\la^{-2} + 1 +\frac{qM}{\la}) q^{2k} + \frac{q^2}{ \la^2}q^{4k} \bigr).
\end{equation*}
The considerations for the positivity of 
$\langle E^n \cdot 1 | E^n \cdot 1\rangle$ and 
$\langle F^n \cdot 1 | F^n \cdot 1\rangle$ for all $n\in\N$ lead
to Theorem \ref{thm:unitarityirrepMathieumodule}. 

\begin{thm}\label{thm:unitarityirrepMathieumodule}
Assume $M(\C_{\la, \mu})$ is an irreducible Mathieu module. 
Then $M(\C_{\la, \mu})$ is unitarizable for the 
$\ast$-structure \eqref{eq:starforUsu11} for $U_q(\Lsu(1,1))$ if and only if
$\la \in \R\setminus\{0\}$ and $\la\mu <0$ and, relabeling $\mu=M/(q-q^{-1})^2$, 
\begin{gather*}
1 - (\la^2 + q^2 + qM\la)x +  q^2\la^2x^2>0, \quad \forall x\in q^{2\N_0}, \\
1 - (q^2\la^{-2} + 1  +\frac{qM}{\la}) x + \frac{q^2}{ \la^2}x^2>0, \quad \forall x\in q^{2\N_0}.
\end{gather*}
In this case, the basis of Proposition \ref{prop:basisMatmodsl2} is orthogonal, with squared norms given by 
$\langle 1|1\rangle =1$ and 
\begin{gather*}
\langle F^n \cdot 1 | F^n \cdot 1\rangle = 
\left( \frac{q^{-1}}{(q-q^{-1})^2} \right)^n \prod_{k=1}^n
\bigl( 1 - (q^2\la^{-2} + 1  +\frac{qM}{\la}) q^{2k} + \frac{q^2 }{\la^2}q^{4k} \bigr), \\
\langle E^n \cdot 1 | E^n \cdot 1\rangle = \left( \frac{q}{(q-q^{-1})^2}\right)^n \prod_{k=1}^{n}   
\bigl(1 - (\la^2 + q^2 + qM\la) q^{2k} + q^2 \la^2 q^{4k}\bigr). 
\end{gather*}
\end{thm}

Note that by putting $A$, $B$, $C$ and $D$ by 
\begin{equation*}
\begin{cases}
AB= q^2\la^2, \\ A+B =q\la M +q^2 +\la^2,
\end{cases}
\qquad 
\begin{cases}
CD= q^2/\la^2, \\
C+D = qM\la^{-1}+1 + q^2\la^{-2}
\end{cases}
\end{equation*}
we can rewrite the positivity condition as
\[
(A,B;q^2)_n>0, \qquad (C,D;q^2)_n >0 \qquad \ \forall n\in \N
\]
using the standard notation for $q$-shifted factorials \cite{GaspR}. 

\begin{proof} It remains to check that the inner product indeed gives a unitary 
representation of $U_q(\Lsu(1,1))$. The relation $K=K^\ast$ is clear, and the 
relation $E^\ast = -FK$ follows by construction. 
\end{proof}

In case the Mathieu modules are reducible, see Proposition \ref{prop:reducibility}, the 
analysis of unitarizability can be done similarly for the 
quotient space. 

\subsection{The irreducible admissible unitary representations of $U_q(\Lsu(1, 1))$}

The representations of $U_q(\Lsu(1, 1))$ have been classified under certain conditions in 
\cite{BurbK}, \cite{MasuMNNSU}, \cite{VaksK}. We restrict to the case of the 
irreducible unitary representation $U_q(\Lsu(1, 1))$ that play an important role
in the harmonic analysis on the quantum group analogue of $SU(1,1)$, see \cite{KoelS}, as well as 
the harmonic analysis on the non-compact quantum group, in the von Neumann algebraic setting,
as the analogue of the normalizer of $SU(1,1)$ in $SL(2,\C)$, see \cite{GroeKK}. 
We restrict to type I admissible representations of $U_q(\Lsu(1, 1))$, where the 
eigenvalues of the action of $K$ are contained in $q^{2\ep+ 2\Z}$, $\ep \in \{0,\frac12\}$. 
Translating the relevant representations we have the following irreducible 
$\ast$-representations of $U_q(\Lsu(1, 1))$, where one should note that the 
representations are given by unbounded operators defined on the domain of the finite 
linear combinations of the basis vectors. The Hilbert space is 
$\ell^2(\N_0)$, respectively $\ell^2(\Z)$, equipped with orthonormal basis $\{e_k\}_{k\in \N_0}$, respectively 
$\{e_k\}_{k\in \Z}$. These representations are classified by the action of the Casimir and the 
eigenvalues of $K$, where the Casimir operator $\Om$ acts
as $(q^{2\sigma +1} +q^{-2\sigma-1})/(q^{-1}-q)^2$ with the value for $\si$ given below 
for the non-extremal unitary representations of $U_q(\Lsu(1, 1))$.

\begin{enumerate}[(i)]
\item \emph{Principal series} acts in $\ell^2(\Z)$. Labeling  $\si = -\frac12 + ib$, 
with $0 \leq b \leq  -(\pi/2 \ln q)$ and $\ep\in \{0, \frac12\}$ and assume 
$(\si, \ep)\not= (-\frac12, \frac12)$.
The eigenvalues of $K$ are $q^{2\ep +2\Z}$.  
\item \emph{Strange series} acts in $\ell^2(\Z)$. Labeling $\ep  \in \{0, \frac12 \}$, 
$\si = -\frac12  - (i\pi/2 \ln q) + a$, $a > 0$.
The eigenvalues of $K$ are $q^{2\ep +2\Z}$. 
\item \emph{Complementary series} acts in $\ell^2(\Z)$. Labeling $-\frac12 <\si < 0$. 
The eigenvalues of $K$ are $q^{2\Z}$.   
\end{enumerate}

The explicit action can be found in e.g. \cite{KoelS}, see  
also \cite{BurbK}, \cite{MasuMNNSU}, \cite{VaksK} for more general representations. 

Upon comparing with Proposition \ref{prop:basisMatmodsl2} and Corollary \ref{cor:lem:basisMatmodsl2}  and Proposition \ref{prop:equivalenceirrMathieumodulessl2} 
we see that the principal series, strange series and complementary series can be matched by considering the suitable $(\la, \mu)$ such that 
the spectrum of $K$, i.e. $\la q^{2\Z}$, and the eigenvalue of
the Casimir match, i.e. 
\[
\mu + \frac{q^{-1} \la  + q \la^{-1}}{(q - q^{-1})^2} = 
\frac{q^{2\sigma +1} +q^{-2\sigma-1}}{(q^{-1}-q)^2},
\]
see Corollary \ref{cor:lem:basisMatmodsl2}. Then, by 
Proposition \ref{prop:equivalenceirrMathieumodulessl2},
all these choices lead to equivalent representations. 
So we recover the principal series, strange series and complementary series representations of 
$U_q(\Lsu(1, 1))$ as irreducible unitary Mathieu modules $M(\C_{\la,\mu})$, where the 
values of $(\la,\mu)$ are not uniquely determined, but the corresponding Mathieu module is. 

Apart from the non-extremal unitary representations, $U_q(\Lsu(1, 1))$ has two sets of 
extremal unitary representations. These are the positive and negative discrete series 
representations;

\begin{enumerate}[(iv)]
\item \emph{Positive discrete series} acts in $\ell^2(\N_0)$. Labeling 
$\si=-k$, $k\in\frac12 \N$, and the eigenvalues of $K$ are 
$q^{2k+2\N}$. 
\item \emph{Negative discrete series} acts in $\ell^2(\N_0)$. Labeling 
$\si=-k$, $k\in\frac12 \N$, and the eigenvalues of $K$ are 
$q^{-2k-2\N}$. 
\end{enumerate}

For the positive and negative series representations we need to take a quotient 
of the Mathieu module, see Corollary \ref{cor:prop:weightmodsarequotientsMathmod}. 
For the positive discrete series we can take $n_F=1$ in 
\eqref{eq:F-reducibility}, so 
$\mu=0$, which corresponds to the degenerate Mathieu module, cf.
Proposition \ref{prop:degMathieumodisVerma}.
Next take $\la = q^{2k}$, and then the positive discrete series is equivalent
to the quotient of the corresponding Mathieu module by the invariant subspace. The negative discrete series can be dealt with by taking $n_E=1$. 
It is well known that these representations are unitary, as can be checked using by performing
the analysis of 
Theorem \ref{thm:unitarityirrepMathieumodule} in case of non-irreducible Mathieu modules.


\section{Rank $1$ Mathieu modules for $U_q(\Lsl(n+1,\C))$}\label{sec:Mathieumodsln+1C}

The setting of Section \ref{sec:Mathieumodsl2C} for the case $\Lg=\Lsl(2,\C)$ is very special, since the 
weight space $U_0$ is a commutative algebra. In this section we consider the case of $U_q(\Lsl(n+1,\C))$ for 
$n\geq 2$, in which $U_0$ is not commutative.

In the setting of Theorem \ref{thm:commsubalgU0} we take 
$S=\{i_1,\cdots, i_s\}$, $s=|S|$, any subset of $\{1,\cdots, n\}$ with the condition that 
$|i_k-i_l|>1$. Then $S$ is a set of strongly orthogonal roots, see Remark \ref{rmk:stronglyorthroots}
and \cite{AgaoK}. 
Then $U_q(\Lg_S)$ corresponds to a product of commuting copies of $U_q(\Lsl(2,\C))$, see 
Remark \ref{rmk:UgS0}.
So $U_q(\Lg_S)$ is  Hopf subalgebras of $U_q(\Lsl(n+1,\C))$ generated by $E_j, F_j, K_j^{\pm1}$, $j \in S$. 
Denote by $U_q(\Lg_S)^+$, respectively $U_q(\Lg_S)^0$, $U_q(\Lg_S)^-$, the subalgebras of $U_q(\Lg_S)$ 
generated by $E_j$, respectively $F_j$, $K_j^{\pm1}$, for $j \in S$. 
In case $S$ consists of one element, we write $U_q(\Lg_{\{j\}})=U_q(\Lg_j)$, and 
then $U_q(\Lg_j)\cong U_q(\Lsl(2,\C))$. 

Using the description of $U_0$ for $U_q(\Lsl(2,\C))$ in Corollary \ref{cor:lem:sl2EnFn=p(EF)} we obtain Lemma \ref{lem:U0sln+1C}.

\begin{lem} \label{lem:U0sln+1C}
$U_0^S$ is the commutative algebra generated by $U^0$ and by $E_jF_j$, $j\in S$. 
\end{lem}

So in particular, for $j\in S$ we have $\C[K_j^{\pm 1}, E_jF_j]\subset U_0^S$. 

\begin{lem} \label{lem:repssl2insln}
If $v$ is an element of a $U_q(\Lsl(n+1,\C))$-module $W$ such that $E_jF_jv = \mu_j v$ and $K_j^{\pm1} = \la_j^{\pm1}v$ for $\mu_j,\la_j \in \C$, 
then the $U_q(\Lg_j)$-module generated by $v$ is isomorphic to a quotient of the $U_q(\Lsl(2,\C))$-module 
$M(\C_{\la_j, \mu_j})$. In particular, if $M(\C_{\la_j, \mu_j})$ is irreducible as 
$U_q(\Lsl(2,\C))$-module, the $U_q(\Lg_j)$-module generated by $v$ is isomorphic to 
the $U_q(\Lsl(2,\C))$-module $M(\C_{\la_j, \mu_j})$.
\end{lem}

\begin{proof}
Apply Lemma \ref{lem:universalproperty} to the case $U=U_q(\Lsl(2,\C))$ with $V=\C v$ and 
$\tilde{V}=\C_{\la_j,\mu_j}$ with $\psi$ mapping $1$ to $v$. Then $\Psi=M(\psi) 
\colon M(\C_{\la_j,\mu_j})\to W_j\subset W$, $W_j= U_q(\Lg_j)\cdot v$ gives the 
$U_q(\Lg_j)$-intertwiner. Then $\Psi$ is surjective, and hence $W_j$ is a quotient of 
$M(\C_{\la_j,\mu_j})$. 
\end{proof}
 
Note that from Lemma \ref{lem:repssl2insln} and \eqref{eq:E-reducibility}, \eqref{eq:F-reducibility}
we can determine when $E_j^m\cdot v=0$ or $F_j^m\cdot v=0$ for some $m\in \N$ in order 
to study the reducibility of the corresponding Mathieu modules for $U_q(\Lsl(n+1,\C))$. 
However, in case the Mathieu module is associated to a set $S$ of strongly orthogonal roots,
the module is always reducible.

\begin{prop}\label{prop:sln+1-reducibleMathieumod}
Let $S\subset \{1,\cdots,n\}$ as above, and let $\la\in \C^n$, $\mu\in\C^s$. Let $\phi^S_{\la,\mu}$ be the corresponding $1$-dimensional 
representation $U_0^S\to \C$ sending $K_i\mapsto \la_i$, $E_jF_j\mapsto \mu_j$ for $j\in S$, and we denote
the extension to $U_0$ by $\phi^S_{\la,\mu}$ as well. 
Let $M(\C^S_{\la,\mu})$ be the corresponding rank $1$ Mathieu module, then 
$M(\C^S_{\la,\mu})$ has a non-trivial invariant subspace.  
\end{prop}

\begin{proof} Let $W$ be the invariant subspace generated by $F_j\otimes 1$ in $U\otimes_{U_0}\C_{\la,\mu}$
for $j\notin S$. 
As in the proof of Proposition \ref{prop:degMathieumodisVerma}, we see that $1\otimes 1 \notin W$, so that 
$W$ is a proper invariant subspace. 
\end{proof}

\begin{rmk}\label{rmk:prop:sln+1-reducibleMathieumod}
The representations constructed in Proposition \ref{prop:sln+1-reducibleMathieumod}
by modding out the maximal proper subspace 
are in general non-extremal modules of $U_q(\Lsl(n+1,\C))$. 
\end{rmk}

We expect that generically the invariant subspace $W$ is the maximal proper subspace, so that 
$M(\C^S_{\la,\mu})/W$ is irreducible.  
A further study of these representations, possibly in relation to the 
results of \cite{FutoKZ}, is needed in order to determine the usefulness in the 
analytic study of the non-compact quantum group analogs of $SU(r,s)$, $r+s=n+1$, and related 
homogeneous spaces. 


\end{document}